\documentclass[10pt,a4wide]{article}
\usepackage{amsmath}
\usepackage{amssymb}
\usepackage{amsthm}
\usepackage{dsfont}
\usepackage{bbm}
\usepackage{graphicx}
\usepackage{psfrag}
\usepackage{hyperref}
\usepackage{color}
\usepackage{enumitem}
\usepackage{mathdots}
\usepackage{pst-all}
\usepackage{pstricks}
\usepackage{pstricks-add}
\usepackage{pst-plot}
\usepackage{pst-infixplot}
\usepackage{pstricks-add}
\usepackage{wrapfig}
\usepackage{tikz}
\usepackage{capt-of}
\usepackage[hyphenbreaks]{breakurl}

\newtheorem{theorem}{Theorem}[section]
\newtheorem{lemma}[theorem]{Lemma}

\newtheorem{corollary}[theorem]{Corollary}
\newtheorem{example}[theorem]{Example}
\newtheorem{definition}[theorem]{Definition}

 \theoremstyle{remark}
 \newtheorem{remark}[theorem]{Remark}

\renewcommand{\Re}{\mathrm{Re\,}}
\renewcommand{\Im}{\mathrm{Im\,}}
\newcommand{\<}{\langle}
\renewcommand{\>}{\rangle}
\renewcommand{\H}{\mathcal{H}}
\newcommand{\BL}{\mathcal{L}}
\newcommand{\R}{\mathbb{R}}		
\newcommand{\N}{\mathbb{N}}	
\newcommand{\C}{\mathbb{C}}	
\newcommand{\F}{\mathbb{F}}	

\newcommand{\D}{\mathcal{D}}
\newcommand{\norm}[1]{\left\| #1 \right\|}	
\newcommand{\abs}[1]{\left| #1 \right|}	
\newcommand{\eps}{\varepsilon}		
\renewcommand{\phi}{\varphi}		

\newenvironment{psmallmatrix}
  {\left[\begin{smallmatrix}}
  {\end{smallmatrix}\right]}

	\newcommand{\ran}{\mathop{\mathrm{ran}}}

\title{Well-posedness of systems of 1-D hyperbolic partial differential equations}

\author{Birgit Jacob\footnote{University of Wuppertal, School of Mathematics and Natural Sciences,
Gau\ss stra\ss e 20,
D-42119 Wuppertal, Germany, bjacob@uni-wuppertal.de}
 \and Julia T.~Kaiser\footnote{University of Wuppertal, School of Mathematics and Natural Sciences,
Gau\ss stra\ss e 20,
D-42119 Wuppertal, Germany. Support by Deutsche Forschungsgemeinschaft (Grants  JA 
735/8-1 
and JA 735/13-1) 
is gratefully acknowledged.~julia.kaiser@uni-wuppertal.de}}
\date{}
\begin{document}
\maketitle

\begin{abstract}
We consider the well-posedness of a class of hyperbolic partial differential equations on a one dimensional spatial domain. 
This class includes in particular infinite networks of transport, wave and beam equations, or even combinations of these. Equivalent conditions 
for contraction semigroup generation are derived. We consider these equations on  a finite interval as well as  on a semi-axis.
\end{abstract}


\section{Introduction}

We consider on an interval $J$ a system of partial differential equations of the following form
\begin{align}\label{eqn:pde}
\frac{\partial x}{\partial t}(\zeta,t) =& \left(\sum_{k=0}^N P_k
  \frac{\partial^k}{\partial \zeta^k} \right) ({\cal H}({\zeta})
x(\zeta,t)),\qquad \zeta \in J, t\ge 0,\\
x(\zeta,0) =& x_0(\zeta),\nonumber
\end{align}
where $P_N$ is an invertible operator on a Hilbert space $H$ 
and  $P_k\in {\cal L}(H)$, $k=0,\cdots, N$,  with $P_k^*=(-1)^{k+1}P_k$, $k=1,\cdots,N$. Here $\BL(H)$ denotes the set of linear bounded operators on $H$.
${\cal H}(\zeta)$ is a positive operator on $H$ for a.e.~$\zeta\in J$  satisfying ${\cal H}, {\cal H}^{-1}\in 
L^{\infty}(J;{\cal L}(H))$. Thereby, the interval $J$ is either a finite interval or a semi-axis. Without loss of 
generality we consider the finite interval $[0,1]$ and the semi-axis $[0,\infty)$.
This class of partial differential equations covers coupled wave and beam equations and in particular infinite networks of 
these equations, that means a network with an infinite number of edges. There has been an enormous development in the study of the Cauchy problem 
\eqref{eqn:pde} in the case of a finite-dimensional Hilbert space $H$ and a finite interval $J$, see for example 
\cite{BastinCoron,Engel13, JZ, GZM, VanDerSchaftMaschke_2002,Villegas_2007, ZGMV} and the references therein. These 
systems are 
also known as port-Hamiltonian systems, 
Hamiltonian partial differential equations or systems of linear conservation laws. In particular, contraction semigroup 
generation has been studied in \cite{GZM,JZ,AJ,Augner,JaMoZw}. In this paper we aim to generalize these results to 
the infinite-dimensional situation and to the 
semi-axis.
In order to guarantee unique solutions of equation~(\ref{eqn:pde}), we have to impose 
boundary conditions, which will be of the  form
\begin{equation}
\label{eq:bc}
   \hat{W}_B (\Phi({\cal H}x))(\cdot,t)  =0.
\end{equation}
In the case of the finite interval $J=[0,1]$, we assume $\hat{W}_B\in {\cal L}(H^{2N}, H^N)$ and that the operator $\Phi$ is  
given by
\begin{align*} \Phi:{\mathcal 
W}^{N,2}(J;H)\rightarrow H^{2N}, \qquad \Phi(x):=[\Phi_1(x) \; \Phi_0(x)]^T, 
\end{align*}
where $\Phi_i(x):=\left[x(i)\; \ldots \;\frac{d^{N-1}x}{d\zeta^{N-1}}(i)\right]^T$ for $i\in \{0,1\}$ and  $ {\mathcal 
W}^{N,2}(J;H)$ denotes the Sobolev space of order $N$.
If $J=[0,\infty)$, then $\hat{W}_B\in {\cal L}(H^N, \tilde H^N)$, where $\tilde H$ is a subspace of $H$, and  $\Phi$ is  
given by
\begin{align*} \Phi:{\mathcal 
W}^{N,2}(J;H)\rightarrow H^{N}, \qquad \Phi(x):=\Phi_0(x).
\end{align*}
Clearly, whether or not equation (\ref{eqn:pde}) possesses unique  and non\--increa\-sing solutions depend on the 
boundary conditions, or equivalently on the operator $\hat{W}_B$. 
The partial differential equation (\ref{eqn:pde}) with the
boundary conditions (\ref{eq:bc}) can be equivalently written as the abstract
Cauchy problem 
\begin{equation*}
   \dot{x} (t) = A {\cal H} x(t),\qquad   x(0)= x_0 ,
\end{equation*}
where $A$ is the linear operator on the Hilbert space $X:=L^2(J;H)$ given  by 
\begin{equation}\label{operatorA}
Ax:=  \sum_{k=0}^N P_k
  \frac{\partial^k x}{\partial \zeta^k}, \qquad x\in \D(A), 
\end{equation}
\begin{equation}
\label{domainA}
\D(A) := \left\{ x\in {\mathcal W}^{N,2}(J;H)\mid  \hat{W}_B\Phi( x) = 0 \right\}.
\end{equation}
We equip $X$ with $\<\cdot,\cdot\>_{L^2}$, the standard scalar product of $L^2(J;H)$. For convenience, we often write 
$\<\cdot,\cdot\>$ instead.

{The aim of this paper is to give equivalent conditions  for the fact that $A {\cal H}$ generates a contraction semigroup on 
$X$ equipped with the norm $\langle \cdot, {\mathcal H}\cdot\rangle$.} If $J=[0,1]$, then under a weak condition,  we show that $A {\cal H}$ generates a contraction semigroup if and only if 
the operator $A$ is dissipative. Moreover, equivalent conditions in terms of the operator  $\hat{W}_B$ are presented. We 
note that the mentioned weak condition is in particular satisfied if the Hilbert space $H$ is finite-dimensional. However, 
even if $H$ is finite-dimensional, our result contains new equivalent conditions for the   contraction semigroup 
characterization \cite{AJ}. For the case $J=[0,\infty)$, the contraction semigroup property has been shown for some 
specific 
examples \cite[I.4.16]{EN}, \cite{MuNoSe16}, and related results can be found in \cite{BK}, \cite{EF}, \cite{KPS}, 
\cite{KS} and \cite{SSVW}. If $J=[0,\infty)$, $N=1$ and 
$H=\mathbb C^d$ or $\mathbb R^d$,  we provide a characterization of the contraction semigroup property of the operator  $A 
{\cal H}$. Again $A {\cal H}$ generates a contraction semigroup  if and only if the operator $A$ is dissipative. 
The main difference to the case $J=[0,1]$ is that the number of boundary conditions depends on $P_1$.
We conclude the paper with some  examples to illustrate our results.

\section{Main results}\label{sec:main}

In this section, we formulate the main results of the paper for both cases $J=[0,1]$ and $J=[0,\infty)$. The proof of all 
theorems and corollaries
are given in Sections \ref{sec:main1} and \ref{ch:proof}.
We define
\begin{align} \label{defQ}  
Q=( Q_{ij})_{ \substack{ 1 \leq i \leq N \\
1 \leq j \leq N} }&= \begin{cases} (-1)^{i-1} P_{i+j-1} & \text{ if }i+j\leq N+1 \\
0 & \text{ else.}\end{cases} 
\end{align}
Clearly, $Q_{ij} \in \BL(H)$, i.e. $ Q \in \BL(H^N)$ and 
\begin{align*}
Q=\begin{bmatrix}  P_1 & P_2 & P_3 & \cdots &P_{N-1} & P_N \\
-P_2 & -P_3& 
-P_4& \cdots & -P_N & 0 \\
P_3&P_4& \iddots &  \iddots&  0 & 0 \\
-P_4& \iddots &  \iddots&\iddots &  & \vdots \\
\vdots & \iddots & \iddots& & & \vdots \\
(-1)^{N-1}P_N& 0 &\cdots &\cdots &\cdots & 0\end{bmatrix}.
\end{align*}
Thus, $Q \in \BL(H^N)$  is a selfadjoint block operator matrix and invertible due to the fact that $P_N$ is invertible.
Let
$$W_B:=\left[\begin{matrix}  W_1 & W_2   
\end{matrix}\right]:=\hat{W}_B\begin{bmatrix}  Q & - Q \\ I & I  
\end{bmatrix}^{-1}\quad \mbox{and} \quad\Sigma := \begin{bmatrix} 0 & I \\ I & 0 
\end{bmatrix}\in{\cal L}(H^N\times H^N),$$
where $W_1, W_2 \in \BL(H^N)$.
Let $P\in \BL(H)$. We call $P$ negative semi-definite, in short $P\leq 0$, if $\<x,Px\>_{H} \leq 0$ for all $x \in H$. We define $\Re P=\frac{1}{2}(P+P^*)$ and $\Im P=\frac{1}{2i}(P-P^*)$. Thus, $P=\Re P+i\Im P$ and $\Re P \leq 0$ if and only if $\<x, \Re P x\>_H=\Re\<x,Px\>_H \leq 0$.

\subsection{Main results for $J=[0,1]$}

In this subsection, we consider the  operator $A{\cal H}$ on the Hilbert space $X=L^2([0,1];H)$, where $H$ is a (possibly
infinite-dimensional) Hilbert space.
\begin{theorem}\label{theo1}  
 Let $A$ be given by \eqref{operatorA}-\eqref{domainA}. Further,
assume 
 \begin{equation}\label{ranbed}
 \ran \,(W_1-W_2)\subseteq  \ran\,(W_1+W_2).
 \end{equation}
Then the following statements are equivalent:
\begin{enumerate}
\item The operator $A{\cal H}$ with domain $$\D(A{\cal H})={\cal 
H}^{-1}\D(A)= \{x\in X\mid {\cal H}x\in  {\mathcal W}^{N,2}([0,1];H)\mbox{ and } \hat{W}_B\Phi(\H  x) = 0\}$$ generates a contraction semigroup on $(X,\langle \cdot,{\cal H} 
\cdot\rangle)$;
\item $A$ is dissipative, that is, $\Re\langle Ax,x\rangle\le 0$ for every $x\in \D(A)$;
\item     $\Re P_0\le 0$, $W_1+W_2$ is injective and $W_B\Sigma W_B^*\ge 0$;
\item   $\Re P_0\le 0$, $W_1+W_2$ is injective and there exists $V\in {\cal L}(H)$ with 
$\|V\|\le 1$ such that
$W_B= \frac{1}{2}(W_1+W_2)\begin{bmatrix} I+V &I-V\end{bmatrix}$;
\item  $\Re P_0\le 0$ and  $u^*Q u- y^*Q y\le 0$  for every $\left[\begin{smallmatrix} 
u \\ y \end{smallmatrix}\right]\in \ker \hat W_B$.
\end{enumerate}
\end{theorem}

\begin{remark}\label{remark}
\begin{enumerate}
\item Condition \eqref{ranbed}  is in general not satisfied: Let $N=1$, $H=\ell^2$ and 
$W_B=\begin{bmatrix}W_1& W_2 \end{bmatrix} \in \BL(\ell^2\times \ell^2,\ell^2)$ with 
$W_1 e_i:= e_{i+1}+e_i$ and $W_2 e_i:= e_{i+1}-e_i$, where $\{e_i\}_{i\in \mathbb N}$ is a  
orthonormal basis of $\ell^2$. Then $\ran (W_1-W_2)= \ell^2$ whereas $e_1\not\in\ran (W_1+W_2)$.

\item We point out that the implications $1 \Rightarrow 2$, $4\Rightarrow 3$, and the equivalence $2\Leftrightarrow 5$ 
hold even without the additional condition \eqref{ranbed}. Moreover, condition \eqref{ranbed} is not needed for the fact that $2$ implies the injectivity of $W_1+W_2$. 

\item We note that $W_B$ is not uniquely determined, only the kernel of $W_B$  is. However, 
 if $W_B$ does not satisfy condition \eqref{ranbed}, then  in general  it is not possible to choose another operator 
instead of $W_B$ with the same 
kernel such that condition (\ref{ranbed}) 
holds.
\item If $H$ is finite-dimensional, then $A{\H}$ has a compact resolvent, see  \cite[Theorem 2.3]{AJ}. However,
in general, $A{\H}$ does not have a compact resolvent. Take for example $N=1$, $P_1=1$, $P_0=0$, $H=\ell^2$, $\hat{W}_B=[  
I 
\,\,\, L ]$ and ${\H}(\zeta)=I_{\ell^2}$.  Here $L$ denotes the left shift on $H$, that is, $Le_i=e_{i+1}$. Thus, $A$ 
generates the  
left shift semigroup on $X=L^2([0,1];\ell^2)$,
which is isometric isomorph to the left shift on $X=L^2[0,\infty)$. However,  $0$ is a spectral point of $A$, but not in the point spectrum. 
\end{enumerate}
\end{remark}

As a corollary of Theorem \ref{theo1} we obtain the well-known contraction semigroup characterization for the case of a 
finite-dimensional Hilbert space $H$, see \cite{AJ}. However, we remark that Conditions 3 and 
4~are new even in the finite-dimensional situation.  

\begin{corollary} \label{theo2}
 Let $A$ be given by \eqref{operatorA}-\eqref{domainA} and assume that
$H$ is finite-dimensional. Then, assertions 1 to 5 in Theorem \ref{theo1} are equivalent, and, moreover, they are 
equivalent to
\begin{enumerate}
\item[6.]  $\Re P_0\le 0$, $W_B$ surjective and $W_B\Sigma W_B^*\ge 0$;
\item[7.]$ \Re P_0\le 0$, $W_B$ surjective  and there exists $V\in {\cal L}(H)$ with 
$\|V\|\le 1$ such that
$W_B= \frac{1}{2}(W_1+W_2)\begin{bmatrix} I+V &I-V\end{bmatrix}$.
\end{enumerate}
\end{corollary}

\begin{remark} 
If $H$ is infinite-dimensional, then in general Conditions 6 and 7 of the previous corollary are not equivalent 
to the fact that $A{\cal H}$ generates a contraction semigroup. Let $H=\ell^2(\N)$, $N\in\mathbb N$, and $P_i$ and ${\cal H}$ are operators satisfying the general assumptions. First, we consider  $W_B=\begin{bmatrix} W_1 & W_2 \end{bmatrix}$ with $W_1:= \frac{3}{2} R$ and $W_2:= \frac{1}{2} R$, where $R$ 
denotes the right shift  on $\ell^2(\N)$. Then 
$\ran(W_1-W_2)= \ran(W_1+W_2)$, $W_1+W_2$ is injective and $W_B\Sigma W_B^*\ge 0$ but 
$W_B$ is not surjective. Thus, $A{\cal H}$ generates a contraction semigroup on  $(X,\langle \cdot,{\cal H} 
\cdot\rangle)$, but Conditions 6 and 7 are not satisfied.
Conversely, for the choice  $W_B=\begin{bmatrix} I-L & -I-L \end{bmatrix}$,  where $L$ 
denotes the left shift  on $\ell^2(\N)$, surjectivity of $W_B$ holds,
$\ran(W_1-W_2)\subseteq \ran(W_1+W_2)$ and $W_B\Sigma W_B^*\ge 0$, but $W_1+W_2$ is not 
injective. Thus, for these boundary conditions the Conditions 6 and 7 of the previous corollary are satisfied, but $A{\cal H}$ does not  generate a 
contraction semigroup on  $(X,\langle \cdot,{\cal H}\cdot\rangle)$.
\end{remark}

Next, we characterize the property of unitary group generation of $A\H$.

\begin{theorem}\label{theo3}
 Let $A$ be given by \eqref{operatorA}-\eqref{domainA}. Further 
assume 
 \begin{equation}\label{ranbed2}
 \ran\,(W_1-W_2)= \ran\,(W_1+W_2).
 \end{equation}
Then the  following statements are equivalent:
\begin{enumerate}
	\item $A\H$ with domain $\D(A{\cal H}):= \{x\in X\mid {\cal H}x\in \D(A)\}={\cal 
H}^{-1}\D(A)$ generates a unitary $C_0$-group on $(X,\langle \cdot,{\cal H} 
\cdot\rangle)$;
	\item $\Re \<Ax,x\>=0$ for every $x \in \D(A)$;
	\item $\Re P_0 = 0$, $ W_1+W_2$ and $- W_1+W_2$ are injective and $W_B\Sigma W_B^*= 0$;
	\item $\Re P_0 = 0$, $ W_1+W_2$ and $- W_1+W_2$ are injective and there exists 
$V\in 
{\cal L}(H)$ with $\|V\|= 1$ such that
$W_B= \frac{1}{2}(W_1+W_2)\begin{bmatrix} I+V &I-V\end{bmatrix}$;
	\item  $\Re P_0 = 0$ and $u^*Q u- y^*Q y= 0$  for every 
$\begin{psmallmatrix} u\\ v \end{psmallmatrix} \in \ker \hat { W}_B$.
\end{enumerate}
\end{theorem}

\begin{corollary}
 Let $A$ be given by \eqref{operatorA}-\eqref{domainA} and assume that
$H$ is finite-dimensional.  Then, assertions 1 to 5 in Theorem \ref{theo3} are equivalent, and, moreover, they are 
equivalent to
\begin{enumerate}
\item[6.]  $\Re P_0= 0$, $W_B$ surjective and $W_B\Sigma W_B^*= 0$;
\item[7.]$ \Re P_0=0$, $W_B$ surjective  and there exists $V\in {\cal L}(H)$ with 
$\|V\|= 1$ such that
$W_B= \frac{1}{2}(W_1+W_2)\begin{bmatrix} I+V &I-V\end{bmatrix}$.
\end{enumerate}
\end{corollary}

\subsection{Main results for $J=[0,\infty)$}

In this subsection, we choose $J=[0,\infty)$, $N=1$ and $H=\mathbb F^d$ with $\F=\R$ or $\F=\C$, that is, we consider 
the operator  $A\H$,   
\begin{align}\label{Aapp}
  A\H x&=P_0\H x+P_1\frac{\partial}{\partial \zeta}(\H x) \text{ with }\\ \label{DAapp}
\D( A\mathcal{H})&=\left\{x \in L^2([0,\infty);\F^d)\mid \H x\in \mathcal W^{1,2}([0,\infty);\F^d), \hat{W}_B(\H 
x(0))=0\right\}
\end{align}
on the space $X=L^2([0,\infty);\F^d)$.
Here  $P_1$ is an invertible Hermitian $d\times d$-matrix, $P_0\in \mathbb F^{d\times d}$, $\hat W_B\in  \mathbb F^{k\times 
d}$ with $k\in\{0,1,\cdots, d\}$ and
${\cal H}(\zeta)\in \mathbb F^{d\times d}$ is positive definite for a.e.~$\zeta\in [0,\infty)$  satisfying ${\cal H}, {\cal 
H}^{-1}\in L^{\infty}([0,\infty); \mathbb F^{d\times d})$. 
Since $P_1$ is an invertible Hermitian matrix, its eigenvalues are real and nonzero.

We denote by 
$n_1$ the number of positive and by $n_2=d-n_1$ the number of negative eigenvalues of $P_1$ 
and  write
\begin{equation}\label{eqndiagonalohneH}P_1=S^{-1}\Delta S=S^{-1}\begin{bmatrix}\Lambda & 0 \\ 0 & \Theta 
\end{bmatrix}S,\end{equation}
with a unitary matrix $S\in \mathbb F^{d\times d}$, a positive definite diagonal matrix $\Lambda \in  \mathbb 
R^{n_1\times 
n_1}$,  and a negative definite diagonal matrix $\Theta \in  \mathbb R^{n_2 
\times n_2}$. We define $\Delta=\begin{psmallmatrix} \Lambda &0 \\ 0 & \Theta\end{psmallmatrix}$.

\begin{theorem}\label{application} 
Assume $A\H$ is given by (\ref{Aapp})-(\ref{DAapp}), $\hat W_B\in  \mathbb F^{k\times d}$ with $k \leq n_2$ has full row 
rank. Then the following statements are equivalent:
\begin{enumerate}
\item $A{\cal H}$ generates a contraction semigroup on $(X,\langle \cdot,{\cal 
H};
\cdot\rangle)$;
\item $\Re\langle Ax,x\rangle\le 0$ for every $x\in \D(A)$;
\item  $\Re P_0\le 0$ and  $y^*P_1y \geq 0$  for every $y \in \ker \hat W_B$;
\item   $\Re P_0\le 0$, $k=n_2$ and $ \hat W_B=B\begin{bmatrix}U & 
I\end{bmatrix}S$, with  $B\in \F^{n_2\times n_2}$ invertible, $U\in \F^{n_2\times n_1}$ and $\Lambda+U^*\Theta U\geq 0$.
\end{enumerate}
\end{theorem}

Further, we are able to characterize the property of unitary group generation in the case $J=[0,\infty)$.

\begin{theorem}\label{application2}
Let $A\H$ be given by (\ref{Aapp})-(\ref{DAapp}), $\hat W_B\in  \mathbb F^{k\times d}$ with $k\leq \min\{n_1,n_2\}$ has 
full row rank. Then the following 
statements are equivalent: 
\begin{enumerate}
\item $A{\cal H}$ generates a unitary $C_0$-group on $(X, \langle \cdot, {\cal H}\cdot\rangle)$;
\item $\Re\,\langle Ax,x \rangle= 0$ for every $x\in \D(A)$;
\item $\Re\,P_0= 0$  and $y^*P_1 y= 0$ for every  $y\in \ker \hat W_B$;
\item $k=n_1=n_2$, $\Re\,P_0= 0$  and $\hat W_B =\begin{bmatrix} U_1 & U_2 \end{bmatrix} S$,
where $U_1, U_2\in \mathbb F^{n_1\times n_1}$ are invertible with $\Lambda +U_1^*U_2^{-*}\Theta 
U_2^{-1} U_1= 0$. 
\end{enumerate}
\end{theorem}

\section{Proofs of the main results: $J=[0,1]$}\label{sec:main1}
Throughout this section we will assume that $J=[0,1]$,  $X=L^2((0,1);H)$, $A$ is given by \eqref{operatorA}-\eqref{domainA}, 
and $W_B$ and $\Sigma$ are defined as in Section \ref{sec:main}.  In order to prove the main statements it is convenient to 
introduce the following linear combinations of the boundary values \cite{GZM}.

\begin{definition} For $x\in {\H}^ {-1}{\mathcal W}^{N,2}((0,1);H)$ we define so called  boundary port variables, namely 
\emph{boundary flow} and \emph{boundary effort}, by
\begin{align} \label{bpv}
\begin{bmatrix}  f_{\partial,\H x} \\ e_{\partial,\H x} \end{bmatrix} 
&:=\frac{1}{\sqrt{2}}\begin{bmatrix}   Q & - Q \\  I&  I \end{bmatrix}\Phi(\H x) = 
R_{ext}\Phi(\H x),
\end{align}
where $Q$ is defined by \eqref{defQ} and $ R_{ext}:=\frac{1}{\sqrt{2}}\left[\begin{smallmatrix}  Q & 
-Q \\ I&  I \end{smallmatrix}\right] \in \BL(H^{2N})$.
\end{definition}

\begin{remark} 
Thanks to the invertibility of $Q$, the operator $R_{ext}$ is invertible.
Thus, we can use the 
boundary port variables to reformulate the domain of the operator $A{\H}$:
\begin{align*}
\D(A{\H}) &= \left\{ x\in X\mid  {\H}x\in {\mathcal W}^{N,2}((0,1);H)\mbox{ and } \hat{W}_B\Phi(\H x) = 0 
\right\}\\&=\left\{ x\in X\mid {\H}x\in  {\mathcal W}^{N,2}((0,1);H) \mbox{ and }  {W}_B \begin{bmatrix}  
f_{\partial,\H x} \\ e_{\partial,\H x} \end{bmatrix}= 0 \right\},
\end{align*}
where $W_B=\hat{W}_{B} R^{-1}_{ext}$.
\end{remark}

Next, we determine the adjoint operator  of $A$. We define $\tilde{ 
Q}=- Q$ and 
\begin{align*}\begin{bmatrix}  \tilde{ f}_{\partial,\H x} \\ \tilde{ e}_{\partial,\H x} 
\end{bmatrix}=\tilde{ R}_{ext}\Phi(\H x) 
\text{ with }\tilde{ R}_{ext}=\frac{1}{\sqrt{2}}\begin{bmatrix}\tilde{ Q} & -\tilde{ Q} \\ 
\ I &  I \end{bmatrix}. \end{align*} 

\begin{lemma} \label{aad2}
The adjoint operator  of the operator $A$ defined in (\ref{operatorA}) with domain 
(\ref{domainA})
and a boundary operator $W_B$ of the form $W_B=S\begin{bmatrix} I+ V &  I- V 
\end{bmatrix}$ where $ S, V \in \BL(H^N)$ and $S$ is injective, is given by
\begin{align} \label{dadt2}
 A^*y&= P_0^*y-\sum_{k=1}^N  P_k \frac{d^k}{d\zeta^k}y, \quad y\in \D(A^*),\\
 \label{dad2t2}
\D( A^*)&=\left\{y \in {\mathcal W}^{N,2}((0,1);H): S\begin{bmatrix}  I+ V^*& I-V^*\end{bmatrix} \begin{bmatrix}  \tilde{ 
f}_{\partial,y} \\ 
\tilde{ e}_{\partial,y} \end{bmatrix}=0 \right\}.
\end{align}
\end{lemma} 
\begin{proof} The statement can be proved in a similar manner as Proposition 3.4.3 in 
\cite{Augner}, where the statement is shown for finite-dimensional Hilbert spaces $H$.
\end{proof}

\begin{definition}\label{A0}
We define the operators $A_0:\D(A_0)\subseteq X \rightarrow X$ and $(A^*)_0:\D((A^*)_0)\subseteq X \rightarrow X$ by
\begin{align*}
A_0x&:= \sum_{k=0}^N P_k
  \frac{\partial^k}{\partial \zeta^k} x, \quad (A^*)_0y:= P_0^*y-\sum_{k=1}^N  P_k \frac{d^k}{d\zeta^k}y\\
\D(A_0) &=\D(A^*_0)=  {\mathcal W}^{N,2}((0,1);H).
\end{align*}
\end{definition}

Remark, that $A_0$ and $(A^*)_0$ are extensions of $A$ and $A^*$, respectively. Integration by parts yields the following lemma.

\begin{lemma} \label{regl} We have for $x \in  {\mathcal W}^{N,2}((0,1);H)$
\begin{align*}
\Re \langle A_0 x,x\rangle  &=\Re \langle  f_{\partial, x}, e_{\partial, x} 
\rangle_{H^{N}}+\Re \langle P_0x,x\rangle\\
&= \Phi_1(x)^*Q\Phi_1(x)-\Phi_0(x)^*Q\Phi_0(x)+\Re \langle P_0x,x\rangle, \\
\Re \langle (A^*)_0 x,x\rangle  &=\Re \langle  \tilde f_{\partial, x}, 
\tilde e_{\partial, x} 
\rangle_{H^{N}}+\Re \langle P_0x,x\rangle\\
&= \Phi_1(x)^*\tilde Q\Phi_1(x)-\Phi_0(x)^*\tilde Q\Phi_0(x)+\Re \langle P_0x,x\rangle.
\end{align*}
\end{lemma}

Furthermore, we need some technical results. First, we give a generalization of the 
technical Lemma 7.3.2 in \cite{JZ} for $N\geq 1$ and arbitrary Banach spaces $Z$.

\begin{lemma}\label{teclemV}
Let $Z$ be a Banach space and $V \in \BL(Z)$. Then it holds
$$\ker \begin{bmatrix} I+ V &  I- V \end{bmatrix}=\ran \begin{bmatrix} I- V \\ - I- V \end{bmatrix},$$
where $\begin{bmatrix} I+ V &  I- V \end{bmatrix} \in \BL(Z\times Z, Z)$ and $\begin{psmallmatrix}I- V \\ - I- V \end{psmallmatrix} \in \BL(Z, Z\times Z)$.
\end{lemma}

\begin{proof}
Assume $\begin{psmallmatrix}x \\ y \end{psmallmatrix} \in \ker \begin{bmatrix} 
I+ V &  I- V \end{bmatrix}$. Thus, it holds
\begin{equation*}\label{null}x+ Vx+y- Vy=0.\end{equation*}
For $l:=\frac{1}{2}(x-y) \in Z$ we get
\[
( I- V)l=\frac{1}{2}(x-y)-\frac{1}{2} V(x-y)=x
\mbox{ and }
(- I- V)l=-\frac{1}{2}(x-y)-\frac{1}{2} V(x-y)=y.
\]
Thus, it follows $\begin{psmallmatrix}x \\ y \end{psmallmatrix} \in \ran 
\left[\begin{smallmatrix} I-  V \\ - I- V \end{smallmatrix}\right].$
Conversely, assume $\begin{psmallmatrix}x \\ y \end{psmallmatrix} \in \ran 
\left[\begin{smallmatrix} I- V \\ - I- V \end{smallmatrix}\right]$. Then, we have
\begin{align*}
\begin{bmatrix} I+ V &  I- V 
\end{bmatrix}\begin{psmallmatrix}x \\ y \end{psmallmatrix}&=\begin{bmatrix} 
I+ V &  I- V \end{bmatrix}\begin{bmatrix} I- V \\ 
- I- V \end{bmatrix}l=0
\end{align*}
for some $l\in Z$
and the lemma is proved.
\end{proof}

\begin{lemma}{\cite[Lemma 2.4]{KurulaZwart}}\label{teclem}
Let $ W=\begin{bmatrix} W_1 &  W_2 \end{bmatrix} \in \BL(H^{2N}, 
H^N)$ such that $ W_1+ W_2$ is  injective and 
 \begin{equation*} 
\label{eig0}\ran ( W_1- W_2)\subseteq \ran( W_1+ 
W_2).\end{equation*}
Then there exist an unique operator $ V \in \BL(H^N)$ such that
\begin{align}
 W=\begin{bmatrix} W_1 &  W_2 
\end{bmatrix}&=\frac{1}{2}( W_1+ W_2)\begin{bmatrix} I+ V & 
 I- V \end{bmatrix} .\label{eig2}
\end{align}
Moreover, 
\begin{equation*}
\ker \begin{bmatrix}  W_1 &  W_2 \end{bmatrix}=\ker \begin{bmatrix} 
I+ V &  I- V \end{bmatrix} \label{eig3},
\end{equation*}
and 
\begin{equation*} \label{eig4}
\begin{bmatrix}  
W_1 &  W_2 \end{bmatrix} \begin{bmatrix} 0 &  I \\  I & 0 
\end{bmatrix} \begin{bmatrix} W_1 &  W_2  \end{bmatrix}^* \geq 0 
\Leftrightarrow  V V^* 
\leq I.
\end{equation*}
\end{lemma}

\begin{lemma}\label{tl3}
Let $A_0$ be defined as in Definition \ref{A0}.
For an arbitrary element $\begin{psmallmatrix}u \\ v \end{psmallmatrix} \in H^N\times H^N$ 
there exists a function $x \in \D(A_0)$ such that $\Phi(x)=\begin{psmallmatrix}u \\ v 
\end{psmallmatrix}$. 
\end{lemma}

\begin{proof}
We give a constructive proof: Consider 
$\begin{psmallmatrix}u 
\\ v \end{psmallmatrix} \in H^N\times H^N$ where
\begin{align*}
u=\begin{bmatrix}u_1 \\ \vdots \\ u_N\\ \end{bmatrix}  \text{ and } v=\begin{bmatrix}v_1 \\ 
\vdots \\ v_N\\ \end{bmatrix},
\end{align*}
with entries $u_1,\ldots,u_N,v_1,\ldots,v_N \in H$. To construct a proper function
$x$, we define two polynomials, $P_u(\zeta)$ and $P_v(\zeta)$, by
\begin{align*}
P_u(\zeta):=\sum_{i=0}^N \frac{u_{i+1}}{i!} (\zeta-1)^{i} \text{ and } 
P_v(\zeta):=\sum_{i=0}^N \frac{v_{i+1}}{i!} \zeta^{i}.
\end{align*}
Furthermore, we define the functions $\phi_0 \in \mathcal{C}^{\infty}[0,1]$ and 
$\phi_1\in \mathcal{C}^{\infty}[0,1]$ such that
$\phi_0|_{[0,\eps]}=0$ and $\phi_0|_{[1-\eps,1]}=1$ and analogously
 $\phi_1|_{[0,\eps]}=1$ and $\phi_1|_{[1-\eps,1]}=0$ hold.
Thus, for
\begin{equation*}
x:=(\phi_0\cdot P_u+\phi_1\cdot P_v) I_{H^N}\in \mathcal{C}^{\infty}([0,1];H^N)\subseteq\D(A_0)
\end{equation*}
we get $\Phi(x)=\begin{psmallmatrix}u \\ v \end{psmallmatrix}$.
\end{proof}

\begin{lemma} \label{adiss}
Let $A$ be defined by  \eqref{operatorA}-\eqref{domainA}. Then $A$ is dissipative if and only if $A-P_0$ is
dissipative and it holds $\Re P_0 \leq 0$.
\end{lemma}
\begin{proof}
``$\Rightarrow$'': Let $A$ be dissipative. Hence, the operator $A-P_0$ is dissipative if $\Re 
P_0 \leq 0$ holds. We will prove $\Re \<P_0z,z \> \leq 0$ for all  $z \in H$:
Let $z\in H$ and $\Psi(\zeta) \in \mathcal{C}^{\infty}_c(0,1)$ with $\zeta \in [0,1]$ 
an arbitrary, scalar-valued function with $\Psi\not \equiv 0$. We define
$$x:=\Psi(\zeta) z \in \mathcal{C}^{\infty}_c((0,1);H)
\subseteq\D(A) $$ 
and it yields, since the derivation equals zero at the boundary,
\begin{align*}
0 \geq \Re\<Ax,x\>_{L^2}&=\Re\<P_0x,x\>_{L^2}=\Re\<P_0\Psi z,\Psi z\>_{L^2}\notag\\
&=\Re \int_0^1 \abs{\Psi(\zeta)}^2\<P_0 z,z\>_{H}d\zeta\notag\\
&=\norm{\Psi}^2_{L^2}\Re\<P_0 z,z\>_{H}.
\end{align*}

``$\Leftarrow$'': We assume $\Re P_0 \leq 0$ and $\Re \<(A-P_0)x,x\>_{L^2} \leq 0$ for all $x 
\in 
\D(A)$. Thus, we get for $x \in \D(A)$
\begin{align*}
\Re\<Ax,x\>=\Re \<(A-P_0)x,x\>_{L^2}+\Re\<P_0x,x\>_{L^2} \leq 0 \;.
\end{align*}\vspace{-7ex}

\end{proof}

We are now in the position to prove the main results for $J=[0,1]$.

\begin{proof}[Proof of Theorem \ref{theo1}]
Without loss of generality we may assume $\H=I$, see  \cite[Lemma 7.2.3]{JZ}.
The implication $1\Rightarrow 2$ follows by the Lumer-Phillips Theorem, c.f.  
\cite[Theorem II.3.15]{Engel2006}, and the equivalence  $3\Leftrightarrow 4$ has been shown in Lemma \ref{teclem}.

Next, we prove the equivalence $2\Leftrightarrow 5$: Lemma \ref{regl} implies for $x\in\D(A)$
\begin{align*}
\Re \langle A  x,x\rangle  &= \Phi_1(x)^*Q\Phi_1(x)-\Phi_0(x)^*Q\Phi_0(x)+\Re \langle P_0x,x\rangle.
\end{align*}
Note that $x\in  {\mathcal W}^{N,2}((0,1);H)$ satisfies $x\in \D(A)$ if and only if $\begin{psmallmatrix} 
\Phi_1(x)\\\Phi_0(x)\end{psmallmatrix}\in \ker \hat W_B$.
This proves the implication  $5\Rightarrow 2$. We now assume that 2 holds. Then Lemma \ref{adiss} shows that $\Re P_0\le 0$ 
and that $A-P_0$  is dissipative, that is,
$$ \Phi_1(x)^*Q\Phi_1(x)-\Phi_0(x)^*Q\Phi_0(x)\le 0$$
for every $x\in  {\mathcal W}^{N,2}((0,1);H)$ satisfying $\begin{psmallmatrix} \Phi_1(x)\\\Phi_0(x)\end{psmallmatrix}\in \ker \hat W_B$.
Further, by Lemma \ref{tl3}, for an arbitrary element $\begin{psmallmatrix}u \\ v \end{psmallmatrix} \in  \ker \hat W_B$ there
exists a function $x \in \D(A)$ such that $\begin{psmallmatrix} \Phi_1(x)\\\Phi_0(x)\end{psmallmatrix}=\begin{psmallmatrix}u \\ v 
\end{psmallmatrix}$.  This proves 5.

Next, we prove the implication $2\Rightarrow 4$:  Lemma \ref{adiss} shows that $\Re P_0\le 0$ and that $A-P_0$  is 
dissipative, that is, using Lemma \ref{regl} 
\begin{equation}\label{g}
\Re \<f_{\partial,x},e_{\partial,x}\>_{H^{N}}\leq 0, \qquad  x \in 
\D(A).
\end{equation}
For an arbitrary element $ \begin{psmallmatrix} f \\ e \end{psmallmatrix} \in \ker W_B 
\subseteq H^{N}\times H^N$ a function $x \in \D(A)$ exists due to Lemma \ref{tl3} such that
$R_{ext}\Phi(x)=\begin{psmallmatrix} f_{\partial,x} \\ e_{\partial,x} 
\end{psmallmatrix}=\begin{psmallmatrix} f \\ e \end{psmallmatrix}$. 
By equation (\ref{g}) we get $e^*f+f^*e \leq 0$ for all
$\begin{psmallmatrix} f \\ e \end{psmallmatrix} \in \ker W_B$, where $W_B:= 
\begin{bmatrix}W_1 &W_2 \end{bmatrix}$. For $y \in \ker(W_1+W_2)$  we have
$W_B\begin{psmallmatrix} y \\ y \end{psmallmatrix}=0$ and thus $y^*y+yy^* \leq 0$.
Since the norm of an element  is non negative, it follows $y=0$ and therefore
$\ker(W_1+W_2)=\{0\}$, which shows  the injectivity of $W_1+W_2$.
Due to this fact,  by Lemma \ref{teclem} there exists  an operator $V$ satisfying \eqref{eig2}. 
It remains to show that $\|V\|\le 1$. Let $l\in H^N$ be arbitrarily. 
By  Lemma \ref{teclemV} we obtain  $\begin{psmallmatrix} I- V \\- 
I- V \end{psmallmatrix} l \in \ker W_B$. 

From Lemma \ref{tl3} we conclude that a function 
$x \in \D(A_0)$ exists, such that $R_{ext}\Phi(x)=\begin{psmallmatrix}  f_{\partial,x} \\ 
  e_{\partial,x}\end{psmallmatrix}=\begin{psmallmatrix} I- V \\- 
I- V \end{psmallmatrix} l $. Therefore, $\begin{psmallmatrix}  f_{\partial,x} \\ 
  e_{\partial,x}\end{psmallmatrix} \in \ker W_B$ and even $x \in \D(A)$.
In conclusion, we obtain using \eqref{g}
\begin{align}\label{ab}
2\Re\< f_{\partial,x}, e_{\partial,x}\>_{H^N}&=\< f_{\partial,x},  
e_{\partial,x}\>_{H^N}+\< e_{\partial,x}, f_{\partial,x}\>_{H^N}\notag
\\&=\<( I- V)l,(- I-V)l\>_{H^N}+\<(- I- V)l,( I-V)l\>_{H^N}\nonumber
\\&=2\<l,(-I+ V^* V)l\>_{H^N}\leq 0
\end{align}
and therefore $\norm{ V}\leq 1$.

Finally, we show the implication  $4\Rightarrow 1$: 
$A$ is a closed operator, see \cite[Lemma 3.2.2]{Augner}. To prove that $A$ generates a contraction 
semigroup, it is sufficient to verify that $A$ and $A^*$ are dissipative; 
c.f.~\cite[Theorem 6.1.8]{JZ}. Let $x\in \D(A)$. Then, 
we have $\begin{psmallmatrix}  f_{\partial,x} \\ 
  e_{\partial,x}\end{psmallmatrix} \in \ker  W_B$ and from Lemma \ref{teclemV}  it follows that there exists a $l\in H^N$ 
such that $\begin{psmallmatrix}  f_{\partial,x} \\
 e_{\partial,x}\end{psmallmatrix}=\begin{psmallmatrix} I- V \\ - 
 I- V \end{psmallmatrix} l$.
Using  Lemma \ref{regl} and Lemma \ref{teclem}, we obtain
\begin{align*}
2\Re\< Ax,x\>_{L^2}&=2\Re\< f_{\partial,x},  
e_{\partial,x}\>_{H^N}+2\langle P_0x,x\rangle\\
&\le 2\<l,(- I+ V^* V)l\>_{H^N}\leq 
0.\end{align*}
Now we consider the adjoint operator $A^*$: Let $y\in \D(A^*)$. By Lemma 
\ref{aad2}, we obtain $\begin{psmallmatrix}\tilde{ f}_{\partial,y}\\\tilde{ 
e}_{\partial,y}\end{psmallmatrix} \in \ker  S\begin{bmatrix}  I+ V^* 
&  I- V^* \end{bmatrix}$.
Applying Lemma \ref{teclemV} and Lemma \ref{teclem} to the operator  $V^*$, there exists $m \in 
H^N$ such that $\begin{psmallmatrix} \tilde{ f}_{\partial,x} \\ 
\tilde{  e}_{\partial,x}\end{psmallmatrix}=\begin{psmallmatrix} I- 
V^* \\ - I- V^* \end{psmallmatrix} m$. Using again Lemma
\ref{regl} we get
\begin{align} \label{ab3}
2\Re\< A^*y,y\>_{L^2}&\le 2\<m,(- I+ V V^*)m\>_{H^N}\leq 
0,
\end{align}
which concludes the proof.
\end{proof}

\begin{proof}[Proof of Corollary \ref{theo2}]
We want to apply Theorem \ref{theo1} for the proof of Corollary \ref{theo2}. Therefore, we have to check condition 
\eqref{ranbed}. If $\dim H < \infty$,  then $W_1+W_2$ injective implies the surjectivity of $W_1+W_2$ and hence condition 
\eqref{ranbed}. Due to this and Remark \ref{remark}.2 assertions 1, 2, 3, 4 and 5 of Theorem \ref{theo2} are equivalent. 
The 
implications $3\Rightarrow 6$ and  $4\Rightarrow 7$ 
follows, since we have $W_1+W_2$ injective, and thus, 
$W_1+W_2$ is also surjective. Clearly, it follows that $W_B$ is surjective.
A straightforward calculation shows the implication $7\Rightarrow 6$. In order to show $6\Rightarrow 3$ we prove that 
in the finite-dimensional setting the surjectivity of $W_B$ and $W_B\Sigma W_B^*\geq 0$  implies the 
injectivity of $W_1+W_2$. From

\begin{equation*}
 W_B\Sigma W_B^*\geq 0 \Leftrightarrow W_2W_1^*+W_1W_2^*\geq 0,
\end{equation*}
we obtain
\begin{equation*}
 W_1W_1^*+W_2W_1^*+W_2W_2^*+W_1W_2^*=(W_1+W_2)(W_1+W_2)^*\geq(W_1-W_2)(W_1-W_2)^*\geq0.
\end{equation*}
Let $x$ be in $\ker(W_1+W_2)^*$. This yields $x \in \ker(W_1-W_2)(W_1-W_2)^*$. With
\begin{align*}
\norm{(W_1-W_2)^*x}^2&=\<(W_1-W_2)^*x,(W_1-W_2)^*x\>\\&=\<x,(W_1-W_2)(W_1-W_2)^*x\>=\<x,0\>=0
\end{align*}
we get $x \in \ker(W_1-W_2)^*$ and thus, $x \in \ker W_1^*\cap W_2^*$. Since $W_B$ is surjective, $W_B^*$  is injective and 
thus $x=0$. This implies that $W_1+W_2$ is injective.
\end{proof}

\begin{proof}[Proof of Theorem \ref{theo3}] Without loss of generality again we consider just the case $\H=I$. In 
the following proof we will often apply Theorem \ref{theo1} to the operators $A$ and $-A$. 
So, first of all, we have to verify, that also the boundary condition operator $\bar{W}_B$ 
of $-A$ satisfies the condition (\ref{ranbed}).

We define analogously to (\ref{bpv}) the boundary flow and the boundary effort for $-A$:
\begin{align}\label{bargrenzfluss}
\begin{bmatrix}  \bar{{f}}_{\partial, x} \\ \bar{ e}_{\partial,x} 
\end{bmatrix} &:=\frac{1}{\sqrt{2}}\begin{bmatrix}  - Q &  Q \\  I& 
 I \end{bmatrix}\Phi(\H x).
\end{align}
Therefore, it yields $\bar{{f}}_{\partial,x}=-{{f}}_{\partial,x}$ and $\bar{ 
e}_{\partial,x}= e_{\partial,x}$. Due to $\D( A)=\D(- A)$, we get
\begin{align*}\D( A)&=\left\{x \in \mathcal W^{N,2}((0,1);H) |  W_B \begin{bmatrix}  
 f_{\partial,x} \\   e_{\partial,x} \end{bmatrix}=0 
\right\}\\&=\left\{x \in\mathcal W^{N,2}((0,1);H) | \bar{ W}_B 
\begin{bmatrix}  \bar{ f}_{\partial,x} \\ \bar{  e}_{\partial,x} 
\end{bmatrix}=0 \right\}\\&=\left\{x \in\mathcal W^{N,2}((0,1);H) | \bar{W}_B 
\begin{bmatrix}  -{ f}_{\partial,x} \\ {  e}_{\partial,x} \end{bmatrix}=0 
\right\}\end{align*}
and thus,
\begin{equation}\label{wminus}\bar{W}_B=\begin{bmatrix}- W_1 &  W_2 
\end{bmatrix}.\end{equation}
It is easy to check that under condition (\ref{ranbed2}) the operator $\bar{W}_B$ satisfies
(\ref{ranbed}).

Then the equivalences $1\Leftrightarrow 2\Leftrightarrow 5$ follow by Theorem
\ref{theo1} applied for $A$ and $-A$.

$1\Rightarrow 4$: Let $A$ be the generator of a unitary group. Then, due to Theorem 
\cite[Theorem 6.2.5]{JZ} $A$ and $-A$ are generators of contraction semigroups. It follows 
$\Re 
P_0=0$, $W_1+W_2$ and $- W_1+W_2$ are injective and $\Re \< Ax,x\>=0\;$ for all $x \in 
\D(A)$ 
by Theorem 
\ref{theo1}. Thus, we get with the estimation (\ref{ab})
\begin{equation} 
 0=2\Re \< Ax,x\>=2\<l,(-I+V^*V)l)\>_{H^N} \text{ for all } \;l \in H^N
\end{equation}
and therefore $\norm{V}=1$.

$4 \Rightarrow 3$: Let $\Re P_0=0$, $\norm{V}=1$,  $ W_1+W_2$ and $- W_1+W_2$  injective. 
Define $S:=\frac{1}{2}(W_1+W_2)$ and with the technical Lemma \ref{teclem} (Lemma 2.4 in 
\cite{KurulaZwart}) it yields 
\begin{align*}
 W_B\Sigma W_B^*&=S\begin{bmatrix}I+V & I-V \end{bmatrix}\begin{bmatrix} 0 & I 
\\ I & 0\end{bmatrix} (S\begin{bmatrix}I+V & I-V \end{bmatrix})^*\\&=S(2I-2VV^*)S^*=0.
\end{align*}
The implication $3 \Rightarrow 1$ follows analogously to the proof of $3
\Rightarrow 1$ in Theorem \ref{theo1} for the operator $-A$. However, instead of the boundary effort and the boundary 
flow for $A$ we need to consider them for $-A$ and have to determine the boundary condition 
operator $\bar{W}_B$ for $-A$.
\end{proof}

\section{Proofs of the main results: $J=[0,\infty)$} \label{ch:proof}

Throughout this section we will assume that $J=[0,\infty)$ and  $A$ is given by \eqref{Aapp}-\eqref{DAapp}. For the proof 
of 
the main statements we need the following technical assertions.

\begin{lemma}\label{posneg}\begin{enumerate}
 \item \label{Lemmapos}
Assume  $\Lambda \in \mathbb \R^{n_1\times n_1}$ is a positive, invertible diagonal matrix and $y\in L^2([0,\infty);\mathbb 
F^{n_1})$. Then the function 
\begin{equation}\label{eqn:Lemmapos}
x(t):= \int_0^\infty e^{-s\Lambda^{-1}} \Lambda^{-1}y(s+t)\, ds, \qquad t\ge 0,
\end{equation}
satisfies $x\in \mathcal W^{1,2}([0,\infty);\mathbb F^{n_1})$ and $x-\Lambda x' =y$.

\item \label{Lemmaneg}
Let $\Theta\in\mathbb R^{n_2\times n_2}$ be a negative, invertible diagonal matrix, $y\in L^2([0,\infty);\mathbb F^{n_2})$ 
and $x_0\in\mathbb F^{n_2}$. Then the differential equation
\begin{equation}\label{eqn:Lemmaneg}
x-\Theta x' =y, \quad x(0)=x_0,
\end{equation}
has a unique solution satisfying  $x\in \mathcal W^{1,2}([0,\infty);\mathbb F^{n_2})$.
\end{enumerate}
\end{lemma}

\noindent
{\bf Proof} Part 1: $\Lambda>0$ and $y\in L^2([0,\infty);\mathbb F^{n_1})$ imply that $x(t)$ is well defined  for every $t\ge 
0$.  
Minkowski's integral inequality shows $x\in L^2([0,\infty);\mathbb F^{n_1})$. Further, the solution of $x-\Lambda x' =y$, or 
equivalently, of $x'=\Lambda^{-1}x-\Lambda^{-1}y$ is given by
\[ x(t)= e^{t\Lambda^{-1}}x(0) -\int_0^t e^{(t-s)\Lambda^{-1}} \Lambda^{-1}y(s)\, ds, \qquad t\ge 0.\]
The choice of $x(0)=  \int_0^\infty e^{-s\Lambda^{-1}} \Lambda^{-1}y(s)\, ds$, implies \eqref{eqn:Lemmapos}. Moreover, 
$x'=\Lambda^{-1}x-\Lambda^{-1}y$ and hence  $x\in \mathcal W^{1,2}([0,\infty);\mathbb F^{n_1})$. \\
Part 2: We first note that \eqref{eqn:Lemmaneg} is equivalent to  $x'=\Theta^{-1}x-\Theta^{-1}y$. Now the statement of 
the lemma follows from ODE-Theory for linear stable systems, since $\Theta<0$ and $y\in L^2([0,\infty);\mathbb 
F^{n_2})$, see \cite[Proposition 3.3.22]{HiPr05}. \hfill${\Box}$\\

\begin{proof}[Proof of Theorem \ref{application}]
Thanks to  \cite[Lemma 7.2.3]{JZ} and the Theorem of Lumer-Phillips 1 implies 2.\\
Next, we show the implication 2$\,\Rightarrow\,$3. Using integration by parts and $P_1^*=P_1$, it yields $2\Re 
\<P_1\frac{d}{d\zeta}x,x\>=-x(0)^*P_1x(0)$, since $\lim_{\zeta\rightarrow \infty}x(\zeta)=0$ for $x\in \mathcal 
W^{1,2}([0,\infty);\mathbb F^d)$. Thus, for $x\in D(A)$ we have
\begin{equation}\label{eq:1HZ}2 {\rm Re}\, \langle Ax, x\rangle =2\Re\<P_1\frac{d}{d\zeta}x+P_0x,x\>= -x(0)^*P_1 x(0) + 
2{\rm Re}\,\int_ 0^\infty 
x(\zeta)^*P_0x(\zeta)\,d\zeta.
\end{equation}
Choosing $x \in{\mathcal  W}^{1,2}([0,\infty);{\mathbb F}^d)\backslash\{0\}$ with $x(0)=0$, we obtain  Re$\, P_0\le 0$.  For 
every $y 
\in 
{\mathbb F}^d$ and every $\varepsilon >0$ there exists a function $x \in \mathcal W^{1,2}([0,\infty);{\mathbb F}^d)$ such 
that $x(0)=y$ 
and the $L^2$-norm of $x$ is less than $\varepsilon$. Choosing this function in equation (\ref{eq:1HZ}) and letting 
$\varepsilon$ go to zero implies the second assertion in 3.\\
In order to prove the implication 3$\,\Rightarrow\,$4, for $x\in D(A)$ we define $\left[\begin{smallmatrix} 
f_1\\f_2\end{smallmatrix}\right]:=Sx(0)$.
Using \eqref{eqndiagonalohneH}, the second condition in 3 can be written as
\begin{equation}
  \label{eq:3}
  \left[ \begin{array}{cc} f_1^* & f_2^* \end{array} \right]
   \left[ \begin{array}{cc} \Lambda & 0 \\ 0 & \Theta \end{array} \right]
   \left[ \begin{array}{c} f_1\\ f_2 \end{array} \right] \geq 0, \quad \mbox{
     for } \left[ \begin{array}{c} f_1\\ f_2 \end{array} \right] \in
   \ker {\hat W}_B S^{-1}.
\end{equation}
Since ${\hat W}_BS^{-1}$ is a full row rank $k\times  d$-matrix with $k\leq n_2$, its kernel has
dimension $d-k$.
By the assumptions on  $\Lambda$ and $\Theta$,  we have $d-k\le n_1$, or equivalently,  $k\ge n_2$. Thus $k=n_2$.

We write $\hat W_BS^{-1}=\begin{bmatrix} U_1 & U_2 
\end{bmatrix}$ with $U_1\in \F^{n_2 \times n_1}$ and $U_2 \in \F^{n_2\times n_2}$. Assuming $U_2$ is not invertible, 
there exists $u\in \mathbb F^{n_2}$ such that $\left[\begin{smallmatrix} 
0\\u\end{smallmatrix}\right] \in \ker \hat{W}_BS^{-1}$ which is  in
 contradiction to \eqref{eq:3}, since $\Theta<0$. Thus, the matrix $\hat W_BS^{-1}$ is of 
the form $B\begin{bmatrix} U & I\end{bmatrix}$, with $U\in \mathbb 
F^{n_2\times n_1}$ and $B\in \mathbb F^{n_2\times n_2}$ invertible. Hence, \eqref{eq:3} is equivalent to
\begin{align}\label{eq:3.1} \left[ \begin{array}{cc} f_1^* & f_2^* \end{array} \right]
   \left[ \begin{array}{cc} \Lambda & 0 \\ 0 & \Theta \end{array} \right]
   \left[ \begin{array}{c} f_1\\ f_2 \end{array} \right] \geq 0 \quad \mbox{
     and }  Uf_1+f_2=0, \mbox{ for } \begin{bmatrix} f_1 \\f_2\end{bmatrix} \in \F^{n_1+n_2}\end{align}
which is equivalent to $\Lambda +U^*\Theta U\ge 0$. This shows 4.\\
It remains to show that 4 implies 1. By  \cite[Lemma 7.2.3]{JZ} it is 
sufficient to prove that $A$ generates a 
contraction semigroup on $(X, \langle \cdot, \cdot\rangle)$.
Due to the fact that Re$\, P_0\le 0$, and  bounded, dissipative perturbations of generators of contraction semigroups, again 
generate a contraction semigroup, see \cite[Theorem III.2.7]{EN}, without loss of generality we may assume $P_0=0$. 

First, we prove the dissipativity of the operator $A$. Let $x\in \D(A)$ and define $\left[\begin{smallmatrix} 
f_1\\f_2\end{smallmatrix}\right]:=Sx(0)$, where the unitary matrix $S$ is given by \eqref{eqndiagonalohneH}. This 
implies $Uf_1+f_2=0$ as $\hat W_B=B \begin{bmatrix} U & I \end{bmatrix}S$. 

Thus, it yields
\begin{align*}
\Re \<Ax,x\>&=-\<x(0),P_1x(0)\>_{\F^d}=-\<x(0),S^{-1}\begin{bmatrix}\Lambda & 0 \\ 0 & \Theta \end{bmatrix}Sx(0)\>_{\F^d}\\
&=-\<Sx(0),\begin{bmatrix}\Lambda & 0 \\ 0 & \Theta \end{bmatrix}Sx(0)\>_{\F^d}=-(f_1^*\Lambda f_1+f_2^*\Theta f_2)\\
&=-(f_1^*\Lambda f_1+f_1^*U^*\Theta Uf_1)\leq 0
\end{align*}
by the last assertion of 4.

Further, thanks to the Theorem of Lumer-Phillips it remains to show that for every $y\in L^2([0,\infty);\mathbb F^d)$ there 
exists $x\in D(A)$ such that $x-Ax=y$. Equivalently, by \eqref{eqndiagonalohneH} it is sufficient to show that for every 
$y_1\in L^2([0,\infty);\mathbb F^{n_1})$ and $y_2\in L^2([0,\infty);\mathbb F^{n_2})$ there exist functions
$x_1\in \mathcal W^{1,2}([0,\infty);\mathbb F^{n_1})$ and $x_2\in \mathcal W^{1,2}([0,\infty);\mathbb F^{n_2})$ such that
\[ x_1-\Lambda x_1' =y_1, \,\,\,  x_2-\Theta x_2'=y_2 \quad \mbox{ and } \quad U x_1(0)+x_2(0)=0.\]
Let $y_1\in L^2([0,\infty);\mathbb F^{n_1})$ and $y_2\in L^2([0,\infty);\mathbb F^{n_2})$ be arbitrarily. Lemma 
\ref{posneg}.\ref{Lemmapos} 
implies the existence of $x_1\in \mathcal W^{1,2}([0,\infty);\mathbb F^{n_1})$ with 
$x_1(0)= \int_0^\infty e^{-s\Lambda^{-1}} \Lambda^{-1}y_1(s)\, ds$ and 
 $x_1-\Lambda x_1' =y_1$. Finally, Lemma \ref{posneg}.\ref{Lemmaneg} shows that there exists a function
$x_2\in \mathcal W^{1,2}([0,\infty);\mathbb F^{n_1})$ with 
$x_2(0)= -Ux_1(0)$  and 
 $x_2-\Theta x_2' =y_2$. This concludes the proof.
\end{proof}

\begin{proof}[Proof of Theorem \ref{application2}]
Since $A\H$ generates a unitary $C_0$-group if and only if $A\H$ and $-A\H$ generate contraction semigroups 
c.f. \cite[Theorem 6.2.5]{JZ}, the equivalence of assertions 1, 2, and 3 follows directly from 
Theorem \ref{application} for $-A\H$ and $A\H$. 

Formulating assertion 4 of Theorem \ref{application} for $-A$ , we get  $\Re (-P_0) \leq 0$, $k=n_1$, 
$${ \hat {W}}_B=\bar B\begin{bmatrix}I & \bar U \end{bmatrix} S
$$ and $\Theta+\bar U^*\Lambda \bar U \leq 0$, where 
$\bar B \in K^{n_1 \times n_1}$ is invertible. 
Thus, assertion 4 of  Theorem \ref{application} for $-A$ and $A$ is equivalent to  $\Re P_0=0$,  $k=n_1=n_2$ and $\hat W_B =\bar B\begin{bmatrix}I & \bar U \end{bmatrix} S= B\begin{bmatrix} U & I 
\end{bmatrix} S$ with $B$ and $\bar B$ invertible. 
It yields $\bar B=BU$ and $B=\bar B \bar U$ with $B, \bar B$ invertible. Therefore, we get $\bar UU=I$ 
and $\bar U, U$ invertible. Thus, we have $\Theta+\bar U^*\Lambda \bar U \leq 0 \Leftrightarrow U^*\Theta U+\Lambda \leq 
0$. Choosing $U_1=BU$ and $U_2=B$ we get the assertion.
\end{proof}

\section{Examples}
 In this section we now illustrate our results by a number of examples.
Networks of discrete partial differential equations on infinite networks are also considered in 
\cite{Delio}.  
Examples with $J=[0,1]$ and a finite-dimensional Hilbert space $H$ can be found in \cite{JZ}.
\begin{example} 
We choose $J=[0,1]$, $H=\ell^2(\N)$ and consider the operator $A$ given by 
\begin{equation}
 Af=\frac{\partial}{\partial \zeta}f
\end{equation}
on the domain
\begin{equation}
\D(A)=\left\{f \in \mathcal{W}^{1,2}((0,1); \ell^2(\N))| \begin{bmatrix}I & -L\end{bmatrix}\Phi(f)=0 \right\}.
\end{equation}
This means that the network is a path graph, see Figure \ref{linegraph}.

\begin{figure}[h!]\centering \begin{tikzpicture}
\draw[fill=black](0,0)circle(1pt);
\draw[fill=black](0.5,0)circle(1pt);
\draw[fill=black](1,0)circle(1pt);
\draw[fill=black](1.5,0)circle(1pt);
\draw[fill=black](2,0)circle(1pt);
\draw[fill=black](2.5,0)circle(1pt);
\draw[fill=black](3,0)circle(1pt);
\draw[fill=black](3.5,0)circle(1pt);
\draw[fill=black](4,0)circle(1pt);
\draw[fill=black](4.5,0)circle(1pt);
\draw(0,0)->(4.5,0);
\draw[dotted, color=red, thick](4.5,0)->(5.5,0);
\end{tikzpicture}\caption{Path graph}
\label{linegraph}
\end{figure}Clearly, $A$ denotes a port-Hamiltonian operator with $N=1$, $P_1=I$, $P_0=0$ 
and
$W_B=\frac{1}{2}\begin{bmatrix}I+L &I-L \end{bmatrix}$. Here $L$ denotes the left shift and $L^*=R$ the right shift, i.e.,
$L: \ell^2(\N) \to \ell^2(\N)$ is defined by $L(x_1,x_2,\ldots)\mapsto (x_2,x_3,\ldots)$ and $R: \ell^2(\N) \to \ell^2(\N)$ 
is given as $R(x_1,x_2,\ldots)\mapsto (0,x_1,x_2,\ldots)$.
Clearly, it yields $ W_1+W_2=I$, and thus, condition \eqref{ranbed} is fulfilled. Therefore, we can apply Theorem 
\ref{theo1} and check assertion 3:
$W_1+W_2$ is injective and
\begin{align*}
W_B \Sigma W_B^*&=\frac{1}{4}\begin{bmatrix}I+L &I-L \end{bmatrix}\Sigma \begin{bmatrix} I+L 
 &I-L \end{bmatrix}^*=\frac{1}{4}\begin{bmatrix}I-L &I+L \end{bmatrix}\begin{bmatrix} I+L^* 
&I-L^* \end{bmatrix}\\&=\frac{1}{4}(( I-L)(I+L^*)+(I+L)( I-L^*))=\frac{1}{4}(2 I-2LL^*)=0.
\end{align*}
Hence, $A$ generates a contraction semigroup. In the finite-dimensional setting we would 
expect that $A$ also generates a unitary $C_0$-group, since $W_B \Sigma W_B^*=0$. 
However, we can apply Theorem \ref{theo3}, since condition \eqref{ranbed2} is fulfilled: $\ran(L)=\ran(I)$, because the 
left shift is surjective. Using assertion 3 of Theorem \ref{theo3}, we can conclude that $A$ does not generate a unitary 
$C_0$-group, since $-W_1+W_2=-L$ and the left shift is not injective. 
\end{example}

\begin{example}
We choose $J=[0,1]$, $H=\ell^2(\N)$ and consider the operator $A$ given by 
\begin{equation}
 Af=\frac{\partial}{\partial \zeta}f
\end{equation}
on the domain
\begin{equation}
\D(A)=\left\{f \in \mathcal{W}^{1,2}((0,1); \ell^2(\N))| \begin{bmatrix}I & T\end{bmatrix}\Phi(f)=0 \right\},
\end{equation}
where $T: \ell^2(\N) \to \ell^2(\N)$ is defined by $$T(x_1,x_2,\ldots)\mapsto 
\frac{1}{2}(-x_3-x_4,-x_5-x_6,-x_7-x_8,\ldots).$$
These boundary conditions imply that the network is a binary tree, see Figure \ref{bintree}. ~\\
\begin{minipage}[t]{0.5\textwidth}
\vspace{0pt} 
\begin{tikzpicture}
[level 1/.style={sibling distance=30mm,level distance=10mm},level 2/.style={sibling distance=15mm,level distance=10mm},level 3/.style={sibling distance=5mm,level distance=5mm},level 4/.style={sibling distance=3mm,level distance=3mm}]
\tikzstyle{every node}=[fill=black,circle,inner sep=1pt]
\node(0) {}
 {
child {node (1) {} 
			child {node (3) {}
							child {node (7) {}
									child[dotted] {node[fill=none,rectangle] (a) {}} 
									child[dotted]{node [fill=none,rectangle] (b) {}}} 
					child{node (8) {}
									child[dotted] {node[fill=none,rectangle] (c) {}} 
									child[dotted]{node [fill=none,rectangle] (d) {}}} } 
		child{node (4) {}
					child {node (9) {}
									child[dotted] {node[fill=none,rectangle] (e) {}} 
									child[dotted]{node [fill=none,rectangle] (f) {}}}
					child{node (10) {}
									child[dotted] {node[fill=none,rectangle] (g) {}} 
									child[dotted]{node [fill=none,rectangle] (h) {}}}}}	
child{node(2) {}
		child {node (5) {}
					child {node (11) {}
									child[dotted] {node[fill=none,rectangle] (e) {}} 
									child[dotted]{node [fill=none,rectangle] (f) {}}}
					child{node (12) {}
									child[dotted] {node[fill=none,rectangle] (e) {}} 
									child[dotted]{node [fill=none,rectangle] (f) {}}}}
		child{node (6) {}
					child {node (13) {}
									child[dotted] {node[fill=none,rectangle] (e) {}} 
									child[dotted]{node [fill=none,rectangle] (f) {}}} 
					child{node (14) {}child[dotted] {node[fill=none,rectangle] (e) {}} 
									child[dotted]{node [fill=none,rectangle] (f) {}}}}}
		}
;
\path  (0) -- (1) node[draw=none, fill=none, midway, above=0.1pt]{$f_1$};
\path  (0) -- (2) node[draw=none, fill=none, midway, above=0.1pt]{$f_2$};
\path  (1) -- (3) node[draw=none, fill=none, midway, left=0.1pt]{$f_3$};
\path  (1) -- (4) node[draw=none, fill=none, midway, right=0.1pt]{$f_4$};
\path  (2) -- (5) node[draw=none, fill=none, midway, left=0.1pt]{$f_5$};
\path  (2) -- (6) node[draw=none, fill=none, midway, right=0.1pt]{$f_6$};
\path  (3) -- (7) node[draw=none, fill=none, midway, left=0.1pt]{$f_7$};
\path  (3) -- (8) node[draw=none, fill=none, midway, right=0.1pt]{$f_8$};
\path  (4) -- (9) node[draw=none, fill=none, midway, left=0.1pt]{$f_9$};
\path  (4) -- (10) node[draw=none, fill=none, midway, right=0.1pt]{$f_{10}$};
\path  (5) -- (11) node[draw=none, fill=none, midway, left=0.1pt]{$f_{11}$};
\path  (5) -- (12) node[draw=none, fill=none, midway, right=0.1pt]{$f_{12}$};
\path  (6) -- (13) node[draw=none, fill=none, midway, left=0.02pt]{$f_{13}$};
\path  (6) -- (14) node[draw=none, fill=none, midway, right=0.02pt]{$f_{14}$};
\end{tikzpicture} 
\vspace{-1em}
\captionof{figure}{Binary tree}
\label{bintree}
\end{minipage}
\begin{minipage}[t]{0.5\textwidth}
\vspace{0pt}
We write $f\in \mathcal{W}^{1,2}((0,1); \ell^2(\N))$ as $f=(f_1,f_2,\ldots)^T$, where 
$f_i\in \mathcal{W}^{1,2}((0,1);\C^d)$ denotes a function on the $i$-th edge of the binary tree. Clearly, $A$ denotes a 
port-Hamiltonian operator with $N=1$, $P_1=I$, $P_0=0$
and $W_B=\frac{1}{2}\begin{bmatrix}I-T &I+T \end{bmatrix}$. 
It yields $ W_1+W_2=I$, and thus, condition \eqref{ranbed} is fulfilled. 
\end{minipage}
~\\
$W_1+W_2$ is injective and $T^*:\ell^2(\N) \to \ell^2(\N)$ is given by $$T^*(x_1,x_2,\ldots)\mapsto 
\frac{1}{2}(0,0,-x_1,-x_1,-x_2,-x_2,\ldots).$$ We obtain
\begin{align*}
W_B \Sigma W_B^*&=\frac{1}{4}(2 I-2TT^*)=\frac{1}{4}I.
\end{align*}
Hence, $A$ generates a contraction semigroup. 
\end{example}

\begin{example}
Let $J=[0,\infty)$ and $A$ be given by (\ref{Aapp})-(\ref{DAapp}).
\begin{enumerate}[leftmargin=2em]
\item Let $P_1<0$, that is, $n_2=d$, and $\Re P_0 \leq 0$. In this 
situation 
 $A{\cal H}$ with domain 
 \[\D(A{\cal H})={\cal 
H}^{-1}\D(A)= \{x\in X\mid {\cal H}x\in  {\mathcal W}^{1,2}([0,\infty);\C^d)\mbox{ and } 
(\H  x)(0) = 0\}\]
 generates a contraction semigroup on $(X,\langle \cdot,{\cal 
H};\cdot\rangle)$.
\item Let  $P_1>0$, that is, $n_2=0$ and $\Re P_0 \leq 0$. Then 
 $A{\cal H}$ with domain 
 \[\D(A{\cal H})={\cal 
H}^{-1}\D(A)= \{x\in X\mid {\cal H}x\in  {\mathcal W}^{1,2}([0,\infty);\C^d)\}
\] 
generates a contraction semigroup on $(X,\langle \cdot,{\cal 
H};
\cdot\rangle)$.
\item An (undamped) vibrating string can be modelled by 
  \begin{align}
  \label{eq:7.1.1}
    \frac{\partial^2 w}{\partial t^2} (\zeta,t) = \frac{1}{\rho(\zeta)} \frac{\partial }{\partial \zeta} \left( T(\zeta) \frac{\partial w}{\partial \zeta}(\zeta,t) \right), \qquad t\ge 0, \zeta\in [0,\infty),
  \end{align}
 where  $w(\zeta,t)$
  is the vertical position of the string at place $\zeta$ and time $t$, $T (\zeta)>0 $ is the Young's modulus
  of the string, and $\rho (\zeta ) >0$ is the mass density, which may vary along the string. We assume that $T$ and $\rho$ are positive functions satisfying $\rho, \rho^{-1}, T, T^{-1} \in L^{\infty}[0,\infty)$.  By choosing the state variables 
$x_1= \rho \frac{\partial w}{\partial t}$ (momentum) and $x_2 = \frac{\partial w}{\partial \zeta}$ (strain), the partial differential equation  (\ref{eq:7.1.1}) can equivalently be written as
\begin{align}
  \nonumber
  \frac{\partial }{\partial t} \begin{bmatrix} x_1(\zeta,t) \\ x_2(\zeta,t) \end{bmatrix} &= \begin{bmatrix} 0 & 1 \\ 1 & 0 \end{bmatrix} \frac{\partial }{\partial \zeta}\left( \begin{bmatrix} \frac{1}{\rho(\zeta)} & 0 \\ 0 & T(\zeta) \end{bmatrix}\begin{bmatrix} x_1(\zeta,t) \\ x_2(\zeta,t) \end{bmatrix} \right) \\
 \label{eq:7.2.3} &= P_1  \frac{\partial }{\partial \zeta}\left( {\cal H}(\zeta)\begin{bmatrix} x_1(\zeta,t) \\ x_2(\zeta,t) \end{bmatrix} \right),
\end{align}
where $P_1=\left[\begin{smallmatrix} 0 & 1\\ 1 & 0 \end{smallmatrix}\right]$ and  ${\cal 
H}(\zeta)=\left[\begin{smallmatrix}\frac{1}{\rho(\zeta)} & 0\\ 0 & T(\zeta)  \end{smallmatrix}\right]$. 

The boundary conditions for (\ref{eq:7.2.3}) are 
\[ \hat W_B ({\mathcal H}x)(0,t)  =0,\]
where $\hat W_B$ is a $k\times 2$-matrix with rank $k\in\{0,1,2\}$,
or equivalently, the partial differential equation \eqref{eq:7.1.1} is equipped with the boundary conditions
\[   \hat W_B \begin{bmatrix}    \frac{\partial w}{\partial t} (0,t) \\ T\frac{\partial w}{\partial \zeta}(0,t)\end{bmatrix} =0.\]
The matrix  $P_1 $ can be factorized as
\[    P_1  = \begin{bmatrix} 1 & -1 \\ 1 & 1 \end{bmatrix} \begin{bmatrix} 1 & 0 \\ 0 & -1 \end{bmatrix} 
\begin{bmatrix} 1/2 & 1/2 \\  -1/2 & 1/2 \end{bmatrix}, \]
This implies $n_2=1$. Thus, by Theorem \ref{application} the corresponding operator 
\begin{align*}
(A{\mathcal H}x)(\zeta) &= \begin{bmatrix} 0 & 1\\ 1 & 0 \end{bmatrix}  \frac{\partial }{\partial \zeta}\left( \begin{bmatrix}\frac{1}{\rho(\zeta)} & 0\\ 0 & T(\zeta)  \end{bmatrix} x(\zeta) \right);\\
D(A{\mathcal H}) &=\left\{x\in  {\cal W}^{1,2}((0,1);\mathbb F^2)\mid \hat W_B ({\mathcal H}x)(0,t) =0\right\},
\end{align*}
generates a  contraction semigroup on $(L^2((0,1);\mathbb C^2),\langle \cdot,{\cal 
H};\cdot\rangle)$ if and only if 
\[   \hat W_B = \frac{b}{2}  \begin{bmatrix} u-1 & u+1 \end{bmatrix}  \]
for $b\in\F\backslash\{0\}$ and $u\in\F$. More precisely,  the
partial differential equation
\begin{align*}
    \frac{\partial^2 w}{\partial t^2} (\zeta,t) &= \frac{1}{\rho(\zeta)} \frac{\partial }{\partial \zeta} \left( T(\zeta) \frac{\partial w}{\partial \zeta}(\zeta,t) \right), \qquad t\ge 0,  \zeta\in [0,\infty),\\
    (u-1)  & \frac{\partial w}{\partial t} (0,t) +(u+1)  T(0)\frac{\partial w}{\partial \zeta} (0,t) =0, \qquad t\ge 0,\\
     \rho(\zeta)\frac{\partial w}{\partial t}(\zeta,0)&= z_0(\zeta), \qquad \zeta\ge 0,\\
     \frac{\partial w}{\partial \zeta}(\zeta,0)&= z_1(\zeta), \qquad \zeta\ge 0,
  \end{align*}
where  $u\in \F$ and $z_0, z_1\in L^2[0,\infty)$,  possesses a unique solution satisfying
\[ \int_0^\infty \rho(\zeta)\left[  \frac{\partial w}{\partial t} (\zeta,t)\right]^2 +  T(\zeta)\left[  \frac{\partial w}{\partial \zeta} (\zeta,t)\right]^2 d\zeta \le \int_0^\infty \frac{z_0^2(\zeta)}{\rho(\zeta)}  +  T(\zeta)z_1^2(\zeta) d \zeta \]
for $t>0$, which means that the energy of the system is nonincreasing.
\end{enumerate}
\end{example}

\end{document}